\documentclass[12pt,reqno]{amsart}
\usepackage{amsmath,amssymb,enumerate}
\usepackage[dvipdfmx]{graphicx}
\usepackage{color}
\usepackage{comment}
\usepackage[all]{xy}
\usepackage{bm}
\usepackage{ascmac}

\newtheorem{tm}{Theorem}[section]
\newtheorem{lm}[tm]{Lemma}

\newtheorem{re}[tm]{Remark}
\newtheorem{df}[tm]{Definition}
\newtheorem{pr}[tm]{Proposition}

\newtheorem{ex}[tm]{Example}

\makeatletter
\newcommand{\subscripts}[3]{%
  \@mathmeasure\z@\displaystyle{#2}%
  \global\setbox\@ne\vbox to\ht\z@{}\dp\@ne\dp\z@
  \setbox\tw@\box\@ne
  \@mathmeasure4\displaystyle{\copy\tw@_{#1}}%
  \@mathmeasure6\displaystyle{{#2}_{#3}}%
  \dimen@-\wd6 \advance\dimen@\wd4 \advance\dimen@\wd\z@
  \hbox to\dimen@{}\mathop{\kern-\dimen@\box4\box6}%
}
\makeatother

\newcommand{\nn}{\nonumber}

\newcommand{\e}{\mathrm{e}}

\newcommand{\dis}{\displaystyle}

\allowdisplaybreaks[4]

\makeatletter
 
 \@addtoreset{equation}{section}
\makeatother

\begin{document}
\title[Monotonic normalized heat diffusion 
for Regular Bipartite Graphs]
{Monotonic Normalized Heat Diffusion for Regular Bipartite Graphs with Four Eigenvalues}

\author[T. Kubo]{Tasuku Kubo}
\author[R. Namba]{Ryuya Namba}
\date{\today}
\address[T. Kubo]{
Graduate School of Science and Engineering,
Ritsumeikan University, 1-1-1, 
Noji-higashi, Kusatsu, 525-8577, Japan}
\email{{\tt{ra0074vh@ed.ritsumei.ac.jp}}}
\address[R. Namba]{Department of Mathematics,
Faculty of Education,
Shizuoka University, 836, 
Ohya, Suruga-ku, Shizuoka, 422-8529, Japan}
\email{{\tt{namba.ryuya@shizuoka.ac.jp}}}


\maketitle

\begin{abstract}
Let $X=(V, E)$ be a finite regular graph and 
$H_t(u, v), \, u, v \in V$, the
heat kernel on $X$.
We prove that, if the graph $X$ is bipartite
and has four distinct Laplacian eigenvalues, 
the ratio 
$H_t(u, v)/H_t(u, u), \, u, v \in V,$
is monotonically non-decreasing 
as a function of $t$. 
The key to the proof is the fact that such a graph is 
an incidence graph of a symmetric 2-design. 
\end{abstract}


\section{{\bf Introduction}}

Let $X=(V, E)$ be a finite, simple and connected graph. 
We denote the
heat kernel on $X$ by 
$H_t(u, v), \, t \ge 0, \, 
u, v \in V$, which is regarded as 
the probability that an $X$-valued 
continuous-time random walk
starting at $u$
reach $v$ at time $t$. 
We put 
$$
r_t(u, v):=\frac{H_t(u, v)}{H_t(u, u)}, \qquad 
t \ge 0, \, u, v \in V,
$$
whose behavior is of interest in the present paper.
Our study is originally motivated by the following question:

\begin{quote}
{\it What kinds of graphs satisfy the 
property that the function 
$t \longmapsto r_t(u, v)$, 
$t \ge 0$,
is monotonically non-decreasing for distinct vertices $u, v$\,?}
\end{quote}

In what follows, we call the property 
the {\it monotonic normalized heat diffusion}
(MNHD, in abbreviation). 
We should emphasize that 
this question is attributed to Peres 
(see Regev--Shinkar \cite{RS}). 
He also conjectured that 
every vertex-transitive graph would satisfy 
the MNHD. 
However, Regev and Shinkar \cite{RS} gave 
a negative answer to this conjecture 
by constructing a finite vertex-transitive graph 
which does not satisfy 
the MNHD. 
Moreover, they also gave an example 
of a not vertex-transitive but regular graph
which does not satisfy  
the MNHD. 
On the other hand, Price \cite{Price} showed 
that all finite abelian Cayley graphs, 
a class of vertex-transitive graphs,
satisfy the 
MNHD. 

Afterwards, Nica showed in \cite{Nica} the following 
by making effective use of 
the spectral decomposition 
of the graph Laplacian matrix. 

\begin{pr}
\label{Prop:Nica}
Every finite graph with three distinct Laplacian eigenvalues satisfies
the MNHD.  
\end{pr}

By taking these related studies into account, we know that whether
the MNHD holds or not
is so sensitive to the 
geometric structures of underlying finite graphs.
Therefore, it seems to be difficult to 
completely determine the class of finite graphs possessing the MNHD. 
As a further problem at this stage, 
it is natural to 
ask whether the similar claim 
to Proposition \ref{Prop:Nica} holds or not when 
the graph Laplacian has 
four distinct eigenvalues, 
which is posed in Nica \cite{Nica} as an open problem. 

Finite  graphs with a few Laplacian eigenvalues 
are well-studied  in 
combinatorics and spectral graph theory. 
It is known that the number of 
distinct Laplacian eigenvalues 
tends to determine the shape of 
a graph in some sense. 
For instance, a finite graph 
with only one Laplacian eigenvalue does not have any edges. Moreover, that with two distinct Laplacian eigenvalues 
is always complete. 
A regular graph with three distinct Laplacian eigenvalues is known to be strongly regular and  there are lots of studies on them. See e.g., Brouwer--Haemers 
\cite[Sect. 9]{BHb} for more details with many references therein.  

Finite graphs with four distinct Laplacian eigenvalues 
have been investigated as well and various examples are known. 
We refer to van Dam \cite{VD}
for several properties and examples of such graphs. 
One of interesting properties for the graphs 
involves with the so-called walk-regularity. 
A finite graph is said to be {\it walk-regular} 
if the number of closed walks of any length 
from a vertex to itself does not 
depend on the choice of the vertex.
Typical examples of walk-regular graphs 
are distance-regular graphs and vertex-transitive graphs.
Moreover, every finite regular graph with 
at most four distinct Laplacian eigenvalues
is known to be walk-regular. 
In view of these circumstances, 
it is worthwhile investigating our problem 
for the case of 
four distinct Laplacian eigenvalues.

The following is a main result of the present paper. 

\begin{tm}
\label{Thm:four-eigenvalues}
Suppose that $X$ is a 
regular bipartite graph 
with four distinct Laplacian eigenvalues. 
Then $X$ satisfies
the MNHD.
\end{tm}

One may wonder if the bipartiteness is imposed 
in the main theorem.
This is due to a highly technical reason. 
As is seen in Section 3, 
in order to complete the proof, 
we need to estimate 
several complicated quantities 
written in terms of the square of 
the graph Laplacian. 
In general cases, it is hard for us to find out
off-diagonal components of 
the graph Laplacian explicitly, 
which makes the estimates difficult. 
However, with bipartiteness, 
such components are completely written 
in terms of some parameters 
associated to a certain characterization 
of such graphs as incidence graphs 
of symmetric 2-designs. 
Hence, we treat the problem under the 
bipartiteness assumption. 

The rest of the present paper is organized as follows:
In Section 2, we recall some known properties 
of finite graphs with four distinct 
Laplacian eigenvalues as well as 
spectral decompositions of the graph Laplacian matrices. 
The incidence graphs of symmetric 2-designs
are also explained. 
The notion together with the bipartiteness 
play a key role in the proof of 
Theorem \ref{Thm:four-eigenvalues}. 
We give a proof of Theorem \ref{Thm:four-eigenvalues} 
by noting a relation between the square of
the graph Laplacian and the bipartiteness in Section 3.
Finally, we give several examples 
of finite non-bipartite graphs 
with four Laplacian eigenvalues satisfying the MNHD
in Section 4.  
We discuss further possible directions of this study as well.

\section{{\bf Preliminaries}}

Throughout the present paper, 
a graph $X=(V, E)$ is always 
finite, simple and connected, 
where $V$ is the set of all vertices and 
$E$ is the set of all unoriented edges. 
For a set $A$, we denote by $|A|$ 
the number of elements in $A$.

\subsection{Graph Laplacians and their eigenvalues}

Let $d$ be a positive integer 
and $X=(V, E)$ a finite $d$-regular graph. 
Suppose that $|V|=n \ge 2$. 
Then the {\it graph Laplacian} $L$ is 
a symmetric $n \times n$-matrix 
$L=(L(u, v))_{u, v \in V}$ defined by 
$L=D-A$, where
$D$ is the degree matrix and
$A$ the adjacency matrix. 
Then we easily have
\begin{equation}
\label{L^2(u, v)}
L^2(u, v)=\begin{cases}
d^2+d & \text{if }u=v \\
-2dA(u, v)
+\dis\sum_{w \neq u, v}A(u, w)A(w, v) &
\text{if }u \neq v
\end{cases}.
\end{equation}
By virtue of the Perron--Frobenius theorem, 
the graph Laplacian $L$ has zero
as its simple and trivial eigenvalue and 
the constant function 1 is the eigenfunction 
corresponding to zero. 
We here emphasize that the spectra of $L$ and 
those of $A$ are compatible in the sense that 
$\lambda \in \sigma(L)$ 
implies  $d-\lambda \in \sigma(A)$ and 
$\mu \in \sigma(A)$ implies $d-\mu \in \sigma(L)$,
respectively. 
Here, $\sigma(L)$ is the set of distinct eigenvalues 
of $L$, often called the spectral set of $L$. 
Since $L$ is symmetric, 
one has the spectral decompositions of $L$ and 
the corresponding heat kernel $H_t$, $t \ge 0$,
that is, 
$$
L=\sum_{\lambda \in \sigma(L)}\lambda P_\lambda, 
\qquad 
H_t:=\e^{-tL}=\sum_{\lambda \in \sigma(L)}
\e^{-t\lambda}P_\lambda,
$$
respectively, 
where $P_\lambda, \, \lambda \in \sigma(L),$ is 
the projection matrix onto the eigenspace 
associated with $\lambda$. 
As mentioned in the previous section, 
the heat kernel $H_t(u, v)$ represents 
the probability that an $X$-valued 
continuous-time random walk
starting at $u$
reach $v$ at time $t$. 
Since $X$ is connected, 
we have $r_0(u, v)=0$ 
and $r_t(u, v) \to 1$ 
as $t \to \infty$
for $u, v \in V$ 
with $u \neq v$. 
Though there is a regular graph 
$X=(V, E)$ such that 
$r_t(u, v) >1$ for some $u, v \in  V$
and some $t>0$, 
we can verify that 
$r_t(u, v) \le 1$ for 
$t \ge 0$ and $u, v \in V$ if
$X$ is vertex-transitive. See Regev--Shinkar \cite[Sect.~1 and Appendix A]{RS}.

\subsection{Finite graphs with four distinct Laplacian eigenvalues}

 

%

Van Dam studied regular graphs 
with four distinct Laplacian eigenvalues in \cite{VD}. 
He classified such graphs into three classes 
in terms of the number of integral eigenvalues
and exhibited interesting examples of such graphs.
Moreover, van Dam and Spence found in \cite{DS} 
all feasible spectra of finite graphs with $|V| \le 30$ and
four Laplacian eigenvalues
by using both theoretical techniques 
and computer results. 
We also refer to Mohammadian--Tayfeh-Rezaie \cite{MT} and Huang--Huang \cite{HHa,HHb} for related results
on finite and particularly regular 
graphs with four Laplacian eigenvalues. 

The following is the classification of 
Laplacian eigenvalues originally given by van Dam. 

\begin{pr}[cf. {van Dam \cite{VD} and van Dam--Spence \cite{DS}}]
\label{Prop:vanDam}
If $X$ is a $d$-regular graph with $|V|=n$
and has four distinct Laplacian eigenvalues, then
only one of these three properties holds.

\vspace{1mm}
\noindent
{\rm (i)} $X$ has four integral Laplacian eigenvalues.

\vspace{1mm}
\noindent
{\rm (ii)} $X$ has two integral Laplacian eigenvalues and two Laplacian eigenvalues 
of the form 
$d-(a \pm \sqrt{b})/2, \, a \in \mathbb{Z}, \, b \in \mathbb{N}$,
with the same multiplicity.

\vspace{1mm}
\noindent
{\rm (iii)} 
$X$ has one integral Laplacian eigenvalue $0$ 
and the other three have the same multiplicity 
$m = (n-1)/3$ and $d = m$ or $d = 2m$. 

\end{pr}

\subsection{{\bf Incidence graph of symmetric 2-design}}


Let $v, b, d, r, \lambda \in \mathbb{N}$.  
Suppose that $M$ is a set 
with $|M|=v$ and
$\mathcal{B}$ is a family of $b$ 
subsets of $M$.
A subset in $\mathcal{B}$ is called a {\it block}. 

\begin{df}[balanced incomplete block design]
The pair $(M, \mathcal{B})$ is called a 
balanced incomplete block design {\rm (BIBD)}
if the following hold.

\vspace{1mm}
\noindent
{\rm (i)} Each subset in $\mathcal{B}$ 
consists of $d$ elements.

\vspace{1mm}
\noindent
{\rm (ii)} For every $x \in M$, there are $r$ blocks 
including $x$. 

\vspace{1mm}
\noindent
{\rm (iii)} For every distinct $x, y \in M$, 
there are $\lambda$ blocks 
including $x, y$. 

\end{df}

A BIBD is also called a 
$(v, b, d, r, \lambda)$-{\it design}
or a 2-{\it design}. 
We note that $bd=vr$ and $\lambda(v-1)=r(d-1)$ 
necessarily hold among the parameters. 
In particular, if $v=b$ (and $r=d$) holds, then the BIBD is called {\it symmetric}. 
If a BIBD is symmetric, then it is also 
called a $(v, d, \lambda)$-design.  
See e.g., Brouwer--Haemers \cite[Sect. 4.8]{BHb}
for more details. 

\begin{df}[incidence graph of a BIBD]
Let $(M, \mathcal{B})$ be a BIBD. 
An {\it incidence graph of a BIBD} is a 
finite graph such that the vertex set is 
$\mathcal{V}=M \sqcup \mathcal{B}$ and 
two distinct vertices $x$ and $y$ are adjacent 
if $x \in M$, $y \in \mathcal{B}$ and $x \in y$. 
\end{df}

By definition, if $x,y \in M$ 
or $x, y \in \mathcal{B}$,
then $x$ and $y$ are never adjacent in the incidence graph
of $(M, \mathcal{B})$. 
This immediately means that an incidence graph
of $(M, \mathcal{B})$ is bipartite with the bipartition 
$M \sqcup \mathcal{B}$. 
There is the following remarkable relation between 
regular bipartite graphs and 
incidence graphs of symmetric 2-designs.

\begin{pr}[cf.~{Brouwer--Haemers \cite[Proposition 15.1.3]{BHb}}]
\label{Prop:incidence}
A regular bipartite graph 
with four distinct Laplacian eigenvalues 
is an incidence graph of a symmetric {\rm 2}-design.
\end{pr}

Moreover, the Laplacian eigenvalues 
of a given incidence graph 
of a symmetric 2-design are explicitly written
in terms of the parameter $(v, d, \lambda)$. 

\begin{pr}[cf.~{van Dam \cite[Sect. 4.1.2]{VD}}]
\label{Prop:incidence2}
The Laplacian eigenvalues of 
an incidence graph of a symmetric 
$(v, d, \lambda)$-design is given by 
\begin{equation}
\label{lambdas}    
\lambda_0=0, \qquad 
\lambda_1=d-\sqrt{d-\lambda}, \qquad 
\lambda_2=d+\sqrt{d-\lambda}, \qquad 
\lambda_3=2d. 
\end{equation}
\end{pr}

\section{{\bf Proof of Theorem \ref{Thm:four-eigenvalues}}}

Let $X=(V, E)$ be a finite, $d$-regular 
and bipartite graph with $|V|=n \in 2\mathbb{N}$
and bipartition $V=V_1 \sqcup V_2$. 
Suppose that the graph Laplacian $L$ of $X$ has precisely 
four distinct eigenvalues.
Since $X$ is $d$-regular and bipartite, 
Propositions \ref{Prop:incidence} and 
\ref{Prop:incidence2} imply that 
$X$ is an incidence graph 
of $(n/2, d, \lambda)$-design, 
where $\lambda=2d(d-1)/(n-2) \in \mathbb{N}$, 
and its four Laplacian eigenvalues
$\lambda_0, \lambda_1, \lambda_2, \lambda_3$ are written as \eqref{lambdas}.

The aim of this section is 
to give a proof of 
Theorem \ref{Thm:four-eigenvalues} 
by making use of the spectral decomposition 
of $L$ and the corresponding heat kernel
$H_t=\e^{-tL}, \, t \ge 0$, 
as well as Propositions \ref{Prop:incidence}
and \ref{Prop:incidence2}. 

In our setting, the heat kernel is given by
$$
H_t=\e^{-tL}=P_0+\e^{-t\lambda_1}P_{\lambda_1}+
\e^{-t\lambda_2}P_{\lambda_2}+\e^{-t\lambda_3}P_{\lambda_3},
\qquad t \ge 0.
$$
Solving  $I=P_0+P_{\lambda_1}+P_{\lambda_2}+P_{\lambda_3}$
and $L^k=\lambda_1^k P_{\lambda_1}+\lambda_2^k P_{\lambda_2}+\lambda_3^k P_{\lambda_3}$, $k=1, 2$, we have 
\begin{equation}
\label{Ps}
\begin{cases}
P_{\lambda_1} 
= C_1
\big\{L^2-(\lambda_2+\lambda_3)L
+\lambda_2\lambda_3(I-P_0)\big\} \\
P_{\lambda_2} 
= C_2
\big\{ L^2-(\lambda_1+\lambda_3)L
+\lambda_1\lambda_3(I-P_0)\big\} \\
P_{\lambda_3} 
= C_3
\big\{ L^2-(\lambda_1+\lambda_2)L
+\lambda_1\lambda_2(I-P_0)\big\} 
\end{cases},
\end{equation}
where $I$ is the identity matrix and 
$C_1, C_2, C_3$ are constants given by 
$$
\begin{aligned}
C_1=
\frac{1}{2\lambda_2\sqrt{d-\lambda}},  
\qquad 
C_2=
-\frac{1}{2\lambda_1\sqrt{d-\lambda}}, 
\qquad
C_3=
\frac{1}{\lambda_1\lambda_2} 
\end{aligned}
$$
respectively. 
Note that $C_1, C_3>0$ and $C_2<0$. 

For $u, v \in V$ with $u \neq v$, 
we define a function 
$h_{u, v} : [0, \infty) \to \mathbb{R}$ by
$$
h_{u, v}(t):=
H_t'(u, v)H_t(u, u) - H_t(u, v)H_t'(u, u), 
\qquad t \ge 0,
$$
where $H_t'$ stands for the derivative of $H_t$
with respect to $t$. 
A direct computation gives us
\begin{align}
h_{u, v}(t) &= 
 \sum_{i=1}^3\frac{\lambda_i}{n}\Delta_i(u, v)\e^{-t\lambda_i}
+(\lambda_2-\lambda_1)
\Delta_{12}(u, v)\e^{-t(\lambda_1+\lambda_2)}\nn\\
&\hspace{1cm}
+(\lambda_3-\lambda_1)
\Delta_{13}(u, v)\e^{-t(\lambda_1+\lambda_3)} \nn\\
&\hspace{1cm}
+(\lambda_3-\lambda_2)
\Delta_{23}(u, v)\e^{-t(\lambda_2+\lambda_3)},
\label{h(t)}
\end{align}
where 
$$
\Delta_i(u, v) = 
P_{\lambda_i}(u, u) - 
P_{\lambda_i}(u, v), \qquad i=1, 2, 3,
$$
and 
$$
\Delta_{ij}(u, v) =
P_{\lambda_i}(u, v)P_{\lambda_j}(u, u) - 
P_{\lambda_j}(u, v)P_{\lambda_i}(u, u)
$$
for $(i, j)=(1, 2), (1, 3),$ and $(2, 3)$.

Suppose that $L(u, v)=-1$ or equivalently $A(u, v)=1$. 
Then, we easily see that $A(u, w)A(w, v)=0$ 
for all $w \neq u, v$ since $X$ is bipartite. 
Therefore, it follows from \eqref{L^2(u, v)} 
that $L^2(u, v)=-2d$. 
Next suppose that $L(u, v)=0$ 
or equivalently $A(u, v)=0$.
There are two possibilities: 
(1) $u, v \in V_1 \text{ or } u, v \in V_2$, and (2)
$u \in V_1,  v \in V_2 \text{ or } 
u \in V_2,  v \in V_1$. 
In the case (1), it is obvious that 
$$
\sum_{w \neq u, v}A(u, w)A(w, v)=\lambda
$$
by the definition of the block design.
Hence, $L^2(u, v)=\lambda$ follows. 
On the other hand, in the case (2), 
we also see that $A(u, w)A(w, v)=0$ for all $w \neq u, v$. 
Therefore, it holds that
$L^2(u, v)=0$. 
Then the set $V \times V$ 
is decomposed as 
$W_0 \sqcup W_1 \sqcup W_2 \sqcup W_3$, 
where $W_0=\{(u, u) \in V \times V\}$ and
$$
\begin {aligned}
W_1&=\{(u, v) \in V \times V \, 
| \, L(u, v)=-1 \text{ and }L^2(u, v)=-2d\}, \\
W_2&=\{(u, v) \in V \times V \, 
| \, L(u, v)=0 \text{ and }L^2(u, v)=\lambda\}, \\
W_3&=\{(u, v) \in V \times V \, 
| \, L(u, v)=0 \text{ and }L^2(u, v)=0\}. \\
\end{aligned}
$$

In order to prove Theorem \ref{Thm:four-eigenvalues}, 
we need several lemmas. 

\begin{lm}
\label{derivative}
For any $u, v \in V$ 
with $u \neq v$, 
the ratio $r_t(u, v)$
has non-negative derivative at 
$t = 0$. Namely, it holds that 
$h_{u, v}(0) \ge 0$. 
\end{lm}

\begin{proof}
Although the proof has been done in Nica \cite[Lemma 3.2]{Nica},
we here give it for the sake of completeness.
It is sufficient to show that 
$$
H_t'(u, v)H_t(u, u) \ge H_t(u, v)H_t'(u, u)
$$
at $t=0$. Since $H_0=I$, we have 
$H_0(u, u)=1$ and $H_0(u, v)=0$. 
Therefore, we are left with checking that 
$H_0'(u, v) \ge 0$. 
It holds that $H_t'=-LH_t$ and, in particular, 
$H_0'=-LH_0=-L$. Hence, we obtain 
$H_0'(u, v)=-L(u, v) \ge 0$, as desired.  
\end{proof}

\begin{lm}
\label{e-mono}
Let $\alpha>0$ and $a_1, a_2 \in \mathbb{R}$. Then,
the function $F(t):=a_1\e^{\alpha t}+a_2\e^{-\alpha t}$ is 
monotonically non-decreasing in $[0, \infty)$ if 
either of the following two conditions hold. 

\vspace{1mm}
\noindent
{\rm (1)} $a_1 \ge 0$ and $a_2 \le 0$. 

\vspace{1mm}
\noindent
{\rm (2)} $0 \le a_2 \le a_1$. 
\end{lm}

\begin{proof}
Due to basic calculus. 
\end{proof}

\begin{lm}
\label{Lem:P's}
It holds that 
\begin{equation}
\label{P123}
\Delta_1(u, v) \ge 0, \quad 
\Delta_2(u, v) \ge 0, \quad 
\Delta_3(u, v) \ge 0
\end{equation}
for any $u, v \in V$ with $u \neq v$. 
\end{lm}

\begin{proof}
It follows from \eqref{Ps} that 
$$
\begin{aligned}
\Delta_1(u, v)
&=C_1
\big\{ d^2+d - L^2(u, v) - (\lambda_2+\lambda_3)
\big(d-L(u, v)\big)+\lambda_2 \lambda_3\big\}, \\
\Delta_2(u, v)
&=C_2
\big\{ d^2+d - L^2(u, v) - (\lambda_1+\lambda_3)
\big(d-L(u, v)\big)+\lambda_1 \lambda_3\big\}, \\
\Delta_3(u, v)
&=C_3
\big\{ d^2+d - L^2(u, v) - (\lambda_1+\lambda_2)
\big(d-L(u, v)\big)+\lambda_1 \lambda_2\big\}.  
\end{aligned}
$$
We here should recall 
$C_1, C_3>0$ and $C_2<0$. 
Suppose that $(u, v) \in W_1$. 
Then, we have 
$$
\begin{aligned}
\Delta_1(u, v) 
&=C_1\big\{ (d^2+d) + 2d
-(3d+\sqrt{d-\lambda})(d+1)
+2d(d+\sqrt{d-\lambda})\big\}\\
&= C_1(d-1)\sqrt{d-\lambda} \ge 0,\\
\Delta_2(u, v) 
&= C_2\big\{(d^2+d) + 2d
-(3d-\sqrt{d-\lambda})(d+1)
+2d(d-\sqrt{d-\lambda})\big\}\\
&= C_2(1-d)\sqrt{d-\lambda} \ge 0,\\
\Delta_3(u, v) 
&= C_3\big\{(d^2+d) + 2d
-2d(d+1)
+(d^2-d+\lambda)\big\}=C_3\lambda \ge 0.
\end{aligned}
$$
Next suppose that 
$(u, v) \in W_2$.
Then, we have 
$$
\begin{aligned}
\Delta_1(u, v) 
&= C_1\big\{(d^2+d) -\lambda
-(3d+\sqrt{d-\lambda})d
+2d(d+\sqrt{d-\lambda})\big\}\\
&=C_1 \sqrt{d-\lambda}(\sqrt{d-\lambda}+d) \ge 0,\\
\Delta_2(u, v) 
&= C_2\big\{(d^2+d) -\lambda
-(3d-\sqrt{d-\lambda})d
+2d(d-\sqrt{d-\lambda})\big\}\\
&=C_2 \sqrt{d-\lambda}(\sqrt{d-\lambda}-d)
\ge 0,\\
\Delta_3(u, v) 
&= C_3\big\{(d^2+d) -\lambda
-2d^2
+(d^2-d+\lambda)\big\}= 0.
\end{aligned}
$$
Finally, suppose that $(u, v) \in W_3$. 
Then, it holds that 
$$
\begin{aligned}
\Delta_1(u, v) 
&= C_1\big\{(d^2+d) 
-(3d+\sqrt{d-\lambda})d
+2d(d+\sqrt{d-\lambda})\big\}\\
&= C_1d(1+\sqrt{d-\lambda}) \ge 0,\\
\Delta_2(u, v) 
&= C_2\big\{(d^2+d) 
-(3d-\sqrt{d-\lambda})d
+2d(d-\sqrt{d-\lambda})\big\}\\
&= C_2d(1-\sqrt{d-\lambda}) \ge 0,\\
\Delta_3(u, v) 
&= C_3\big\{(d^2+d) 
-2d^2
+(d^2-d+\lambda)\big\}=C_3\lambda \ge 0,
\end{aligned}
$$
where we used $d-\lambda \ge 1$ 
for the estimate of $\Delta_2(u, v)$. 
We also note that the case $d=\lambda$ never occurs in our setting. 
In fact, if $d=\lambda$, then 
the number of Laplacian eigenvalues would be three. 
Therefore, Inequalities \eqref{P123} are established
for all $u, v \in V$ with $u \neq v$. 
\end{proof}

\begin{lm}
\label{Lem:P's2}
The following hold:

\vspace{2mm}
\noindent
{\rm (1)} If $(u, v) \in W_1$, we have
$\Delta_{12}(u, v) \ge 0$, $\Delta_{13}(u, v) \ge 0$ and 
$\Delta_{23}(u, v) \ge 0$. 

\vspace{1mm}
\noindent
{\rm (2)} If $(u, v) \in W_2$, we have
$\Delta_{12}(u, v)=0$, $\Delta_{13}(u, v) \le 0$ and 
$\Delta_{23}(u, v) \le 0$.

\vspace{1mm}
\noindent
{\rm (3)} If $(u, v) \in W_3$, we have 
$\Delta_{12}(u, v) \le 0$, $\Delta_{13}(u, v) \ge 0$ and 
$\Delta_{23}(u, v) \ge 0$. 

\end{lm}

\begin{proof}
It follows 
directly from \eqref{Ps} that 
\begin{align}
 \Delta_{ij}(u, v)
 &=C_iC_j(\lambda_j-\lambda_i)\Big[
 \big\{L^2(u,v)d-L(u,v)(d^2+d)\big\}\nn\\
&\hspace{1cm}+\frac{1}{n}\lambda_k
\Big\{\big(L^2(u,v)-(d^2+d)\big)
-\lambda_k\big(L(u,v)-d\big)\Big\}\nn\\
&\hspace{1cm}-\lambda_k\big(L^2(u,v)
-\lambda_kL(u, v)\big)\Big]
\end{align}
for $(i, j, k)=(1, 2, 3), (1, 3, 2), (2, 3, 1)$.
Suppose that $(u, v) \in W_1$. Then 
it follows from $L(u, v)=-1$ and $L^2(u, v)=-2d$ that
$$
\begin{aligned}
\Delta_{12}(u, v)
&=-d(d-1)C_1C_2(\lambda_2-\lambda_1)
\Big(1-\frac{2d}{n}\Big)\ge 0,\\
\Delta_{13}(u, v)
&=C_1C_3(\lambda_3-\lambda_1)
\Big[\lambda+\frac{1}{n}(d+\sqrt{d-\lambda})
\big\{-2d+(1+d)\sqrt{d-\lambda}\big\}\Big]\\
&=C_1C_3(\lambda_3-\lambda_1)
\Big\{\lambda - \frac{2}{n}(d^2-d+\lambda)
+\frac{1}{n}(d-1)\sqrt{d-\lambda}(d+\sqrt{d-\lambda})\Big\}\\
&=\frac{1}{n} C_1C_3(\lambda_3-\lambda_1)
(d-1)\sqrt{d-\lambda}(d+\sqrt{d-\lambda}) \ge 0,\\
\Delta_{23}(u, v)
&=C_2C_3(\lambda_3-\lambda_2)
\Big[\lambda-\frac{1}{n}(d-\sqrt{d-\lambda})
\big\{2d+(d+1)\sqrt{d-\lambda}\big\}\Big]\\
&=C_2C_3(\lambda_3-\lambda_2)
\Big\{\lambda - \frac{2}{n}(d^2-d+\lambda)\\
&\hspace{1cm}-\frac{1}{n}(d-1)\sqrt{d-\lambda}(d-\sqrt{d-\lambda})\Big\}\\
&=-\frac{1}{n} C_2C_3(\lambda_3-\lambda_2)
(d-1)\sqrt{d-\lambda}(d-\sqrt{d-\lambda}) \ge 0,
\end{aligned}
$$
where we used $d^2-d+\lambda=n\lambda/2$
for the calculations of $\Delta_{13}(u, v)$
and $\Delta_{23}(u, v)$. 
Next we suppose $(u, v) \in W_2$. 
Then, it follows from 
$L(u, v)=0, \, L^2(u, v)=\lambda$
and $d^2-d+\lambda=n\lambda/2$ that  
$$
\begin{aligned}
\Delta_{12}(u, v)
&=C_1C_2(\lambda_2-\lambda_1)
\Big\{\lambda d+\frac{2d}{n}(d^2-d+\lambda)
-2\lambda d\Big\}
=0,\\
\Delta_{13}(u, v)
&=C_1C_3(\lambda_3-\lambda_1)
\Big\{-\lambda\sqrt{d-\lambda} 
+\frac{1}{n}(d+\sqrt{d-\lambda})
(\lambda-d+d\sqrt{d-\lambda})\Big\}\\
&=C_1C_3(\lambda_3-\lambda_1)\sqrt{d-\lambda}
\Big\{-\lambda  
+\frac{1}{n}(d^2-d+\lambda)\Big\}\\
&=-\frac{1}{2}\lambda 
C_1C_3(\lambda_3-\lambda_1)\sqrt{d-\lambda}\le 0,\\
\Delta_{23}(u, v)
&=C_2C_3(\lambda_3-\lambda_2)
\Big\{\lambda\sqrt{d-\lambda} 
+\frac{1}{n}(d-\sqrt{d-\lambda})
(\lambda-d-d\sqrt{d-\lambda})\Big\}\\
&=C_2C_3(\lambda_3-\lambda_2)\sqrt{d-\lambda}
\Big\{\lambda  
-\frac{1}{n}(d^2-d+\lambda)\Big\}\\
&=\frac{1}{2}\lambda 
C_2C_3(\lambda_3-\lambda_2)\sqrt{d-\lambda}\le 0.
\end{aligned}
$$
We finally suppose that $(u, v) \in W_3$. 
Then, by noting 
$L(u, v)=L^2(u, v)=0$, we have
$$
\begin{aligned}
\Delta_{12}(u, v)
&=\frac{2}{n}C_1C_2(\lambda_2-\lambda_1)d^2(d-1)\le 0,\\
\Delta_{13}(u, v)
&=\frac{1}{n}C_1C_3d(\lambda_3-\lambda_1)
(d+\sqrt{d-\lambda})(\sqrt{d-\lambda}-1) \ge 0,\\
\Delta_{23}(u, v)
&=-\frac{1}{n}C_2C_3d(\lambda_3-\lambda_2)
(d-\sqrt{d-\lambda})(\sqrt{d-\lambda}+1)
\ge 0.
\end{aligned}
$$
Putting it all together, 
we have established Lemma \ref{Lem:P's2}. 
\end{proof}

We are in a position to give a proof of 
Theorem \ref{Thm:four-eigenvalues}.

\begin{proof}
Let $u, v \in V$ with $u \neq v$. 
It is sufficient to show that 
$h_{u, v}(t) \ge 0$ for all $t \ge 0$. 
We suppose that $(u, v) \in W_1$. 
Then Lemmas \ref{Lem:P's} and \ref{Lem:P's2}
readily implies that 
$h_{u, v}(t) \ge 0$ for $t \ge 0$. 
Next we suppose that $(u, v) \in W_2$. 
Since it holds that $\Delta_{12}(u, v)=0$
by Lemma \ref{Lem:P's2}, we have 
$$
\begin{aligned}
\e^{\lambda_3t}h_{u, v}(t) 
&=\frac{\lambda_3}{n}\Delta_3(u, v)
+\lambda_2\Big\{\frac{1}{n}\Delta_2(u, v)\e^{t\lambda_1}
+\Delta_{13}(u, v)\e^{-t\lambda_1}\Big\}\\
&\hspace{1cm}+\lambda_1\Big\{\frac{1}{n}
\Delta_1(u, v)\e^{t\lambda_2}
+\Delta_{23}(u, v)\e^{-t\lambda_2}\Big\}, 
\qquad t \ge 0.
\end{aligned}
$$
Note that 
$\frac{1}{n}\Delta_i(u, v) \ge 0$ and 
$\Delta_{i3}(u, v) \le 0$ for $i=1, 2$
by using Lemmas \ref{Lem:P's} and \ref{Lem:P's2}. 
Then, Lemma \ref{e-mono} implies that the function 
$\e^{\lambda_3t}h_{u, v}(t)$ is monotonically non-decreasing 
in $[0, \infty)$ and Lemma \ref{derivative} yields 
$\e^{\lambda_3t}h_{u, v}(t) \ge h_{u, v}(0) \ge 0$, $t \ge 0$. 
Hence, we know that $h_{u, v}(t) \ge 0$ for all $t \ge 0$. 

Finally consider the case where $(u, v) \in W_3$. 
Then we also have
$$
\begin{aligned}
&\e^{\lambda_3t}h_{u, v}(t) \\
&=\frac{\lambda_3}{n}\Delta_3(u, v)
+(\lambda_2-\lambda_1)\Delta_{12}(u, v)
+\lambda_2\Big\{\frac{1}{n}\Delta_2(u, v)\e^{t\lambda_1}
+\Delta_{13}(u, v)\e^{-t\lambda_1}\Big\}\\
&\hspace{1cm}+\lambda_1\Big\{\frac{1}{n}\Delta_1(u, v)\e^{t\lambda_2}
+\Delta_{23}(u, v)\e^{-t\lambda_2}\Big\}, 
\qquad t \ge 0.
\end{aligned}
$$
Since $C_2=-\lambda_2^2C_1C_3$, we have 
$$
\begin{aligned}
&\frac{1}{n}\Delta_2(u, v)-
\Delta_{13}(u, v)\nn\\
&=\frac{1}{n}
C_2d(1-\sqrt{d-\lambda})
- \frac{1}{n}C_1C_3d(\lambda_3-\lambda_1)
(d+\sqrt{d-\lambda})(\sqrt{d-\lambda}-1)\nn\\
&=-\frac{1}{n}C_1C_3\lambda_2 d\Big\{
\lambda_2(1-\sqrt{d-\lambda})
+(d+\sqrt{d-\lambda})(\sqrt{d-\lambda}-1)
\Big\}=0.
\end{aligned}
$$
Similarly, one has
$$
\frac{1}{n}\Delta_1(u, v)-
\Delta_{23}(u, v)=0. 
$$
By virtue of Lemmas \ref{e-mono}, \ref{Lem:P's} and \ref{Lem:P's2}, 
we conclude that the function 
$\e^{\lambda_3t}h_{u, v}(t)$ is monotonically non-decreasing 
in $[0, \infty)$. Hence, we can show that 
$h_{u, v}(t) \ge 0$ for all $t \ge 0$
in the same way as the case where 
$(u, v) \in W_2$. 
\end{proof}


\section{{\bf Examples and further directions}}

Our main result asserts that 
a finite graph $X$ with four Laplacian eigenvalues  
satisfies the MNHD 
if $X$ is regular and bipartite.  
An example of such graphs is given below.

\begin{ex}[regular, bipartite case]
\normalfont
Let $M=\{1, 2, 3, 4, 5, 6, 7\}$ and 
$$
\begin{aligned}
\mathcal{B}&=
\big\{ \{1, 2, 3, 4\}, \, \{1, 2, 5, 6\}, \,
\{1, 4, 6, 7\},  \\
&\hspace{1cm}\{1, 3, 5, 7\}, \, \{2, 3, 6, 7\}, \,
\{2, 4, 5, 7\}, \, \{3, 4, 5, 6\}\big\}. 
\end{aligned}
$$
Then we easily verify that 
$(M, \mathcal{B})$ is a BIBD
with parameter $(7, 4, 2)$. 
Hence, the incidence graph of symmetric
$(7, 4, 2)$-design 
is a 4-regular bipartite graph with bipartition 
$M \sqcup \mathcal{B}$ 
and four Laplacian eigenvalues 
$0, 4 \pm \sqrt{2}, 8$
(see Fig. \ref{bip} below).

\begin{figure}[ht]
\begin{center}
\includegraphics[width=10cm]{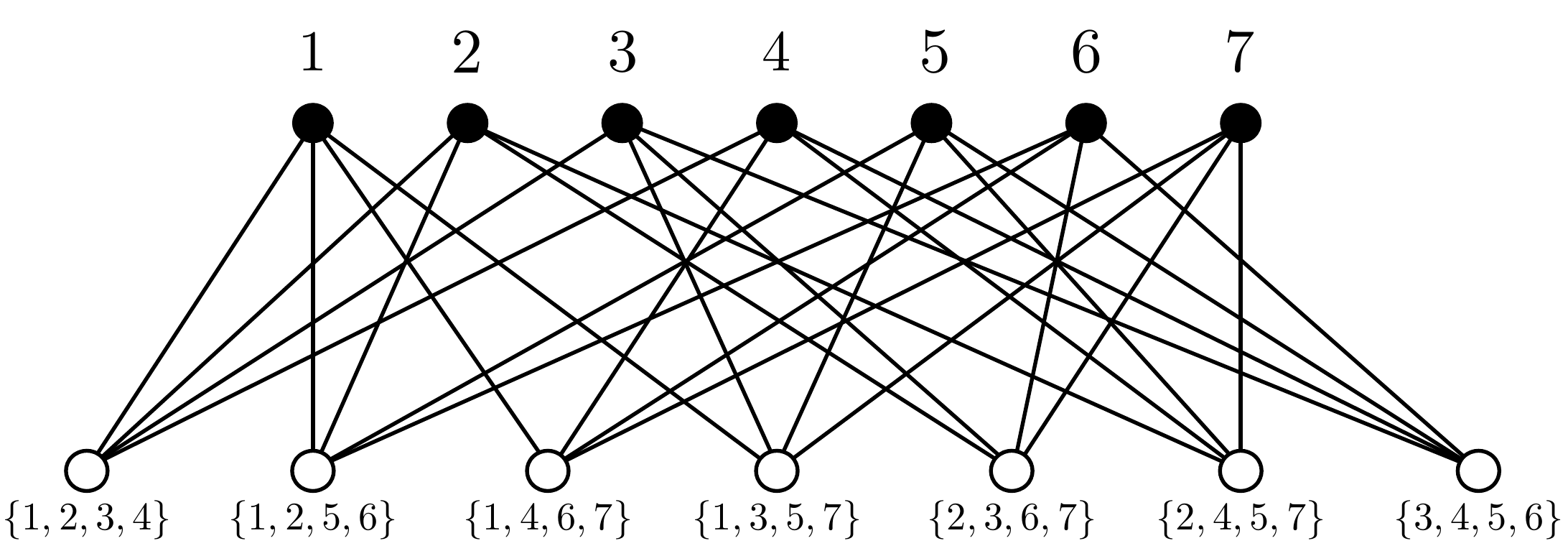}
\end{center}
\caption{The incidence graph of 
$(7, 4, 2)$-design.}
\label{bip}
\end{figure}

In view of Theorem \ref{Thm:four-eigenvalues}, 
this graph satisfies 
the MNHD. 
Note that the graph above 
corresponds to the case (ii) in 
Proposition \ref{Prop:vanDam}.
We also exhibit below all regular bipartite 
graphs $X=(V, E)$ 
with $|V| \le 30$ 
having four distinct Laplacian eigenvalues, 
according to van Dam--Spence \cite[Appendix A]{DS}. 
\begin{center}
\vspace{3mm}
\begin{tabular}{ccc} \hline
$|V|$ &  
$\sigma(L)$ 
& $(v, k, \lambda)$ \\[5pt]
10  & $\{0, 3, 5, 8\}$ &  $(5, 4, 3)$ \\[5pt]
12  & $\{0, 4, 6, 10\}$ &  $(6, 5, 4)$ \\[5pt]
14  & $\{0, 5, 7, 12\}$ &  $(7, 6, 5)$ \\[5pt]
14  & $\{0, 3 \pm \sqrt{2}, 6\}$ 
&  $(7, 3, 1)$ \\[5pt]
14  & $\{0, 4\pm \sqrt{2}, 8\}$ 
&  $(7, 4, 2)$ \\[5pt]
16  & $\{0, 6, 8, 14\}$ &  $(8, 7, 6)$ \\[5pt]
18  & $\{0, 7, 9, 16\}$ &  $(9, 8, 7)$ \\[5pt]
20  & $\{0, 8, 10, 18\}$ &  $(10, 9, 8)$ \\[5pt]
22  & $\{0, 9, 11, 20\}$ &  $(11, 10, 9)$ \\[5pt]
22  & $\{0, 5 \pm \sqrt{3}, 10\}$ 
&  $(11, 5, 2)$ \\[5pt]\hline
\end{tabular}\hspace{0.2cm}
\begin{tabular}{ccc} \hline
$|V|$ &  
$\sigma(L)$ 
& $(v, k, \lambda)$ \\[5pt]
22  & $\{0, 6 \pm \sqrt{3}, 12\}$ 
&  $(11, 6, 3)$ \\[5pt]
24  & $\{0, 10, 12, 22\}$ &  $(12, 11, 10)$ \\[5pt]
26  & $\{0, 11, 13, 24\}$ &  $(13, 12, 11)$ \\[5pt]
26  & $\{0, 4 \pm \sqrt{3}, 8\}$ 
&  $(13, 4, 1)$ \\[5pt]
26  & $\{0, 9 \pm \sqrt{3}, 18\}$ 
&  $(13, 9, 6)$ \\[5pt]
28  & $\{0, 12, 14, 26\}$ &  $(14, 13, 12)$ \\[5pt]
30  & $\{0, 5, 9, 14\}$ &  $(15, 7, 3)$ \\[5pt]
30  & $\{0, 6, 10, 16\}$ &  $(15, 8, 4)$ \\[5pt]
30  & $\{0, 13, 15, 28\}$ &  $(15, 14, 13)$ \\[5pt]
  &   &    \\[5pt]\hline
\end{tabular}
\end{center}

\end{ex}

On the other hand, we also find some examples 
of finite (not always regular)
graphs with four Laplacian eigenvalues 
which are not bipartite but satisfy 
the MNHD. 

\begin{ex}[regular, non-bipartite case]
\label{Ex:1}
\normalfont
Let $X$ be a finite 3-regular graph
given by {\bf (a)} in Fig.~\ref{fig}. 
The graph Laplacian of $X$ has 
four integral eigenvalues 
$0, 2, 3, 5$. 
This means that the graph $X$
corresponds to the case {\rm (i)} in 
Proposition \ref{Prop:vanDam}.
We can show that the graph $X$ 
actually satisfies the MNHD. 
\begin{proof}
The graph Laplacian $L$ of $X$ and 
its square $L^2$ 
are given by 
$$
L=\begin{pmatrix}
3 & -1 & -1 & 0 & 0 & -1 \\
-1 & 3 & 0 & -1 & -1 & 0 \\
-1 & 0 & 3 & -1 & 0 & -1 \\
0 & -1 & -1 & 3 & -1 & 0 \\
0 & -1 & 0 & -1 & 3 & -1 \\
-1 & 0 & -1 & 0 & -1 & 3 \\
\end{pmatrix}, \qquad 
L^2=\begin{pmatrix}
12 & -6 & -5 & 2 & 2 & -5 \\
-6 & 12 & 2 & -5 & -5 & 2 \\
-5 & 2 & 12 & -6 & 2 & -5 \\
2 & -5 & -6 & 12 & -5 & 2 \\
2 & -5 & 2 & -5 & 12 & -6 \\
-5 & 2 & -5 & 2 & -6 & 12 \\
\end{pmatrix},
$$
respectively. 
We decompose the set 
$(V \times V) \setminus 
\{(v_i, v_i) \, | \, i=1, 2, \dots, 6\}$
into $W_1 \sqcup W_2 \sqcup W_3$, where 
$$
\begin{aligned}
W_1&:=\{(v_i, v_j) \, | \, 
\text{$v_i$, $v_j$ are not adjacent}\},\\
W_2&:=\{(v_i, v_j) \, | \, 
\text{$v_i$, $v_j$ are adjacent
and $v_i, v_j, v_k$ form a triangle for some $k$}\},\\
W_3&:=\{(v_i, v_j) \, | \, 
(v_i, v_j) \notin W_1 \sqcup W_2\}. 
\end{aligned}
$$

Then the values of $\Delta_k=\Delta_k(v_i, v_j), \, k=1, 2, 3$, 
and $\Delta_{k\ell}=\Delta_{k\ell}(v_i, v_j)$, 
$(k, \ell)=(1, 2), (1, 3), (2, 3)$, are exhibited as follows:

\vspace{3mm}
\begin{center}
\begin{tabular}{ccccccc}\hline
 & $\Delta_1$    & $\Delta_2$
 & $\Delta_3$    & $\Delta_{12}$
 & $\Delta_{13}$ & $\Delta_{23}$ \\[5pt]
$(v_i, v_j) \in W_1$ &  $\frac{1}{3}$&  $\frac{1}{2}$&  $\frac{1}{6}$&  $-\frac{1}{36}$&  $-\frac{1}{12}$&  $-\frac{1}{9}$  \\[5pt]
$(v_i, v_j) \in W_2$ &  $0$&  $\frac{1}{2}$&  $\frac{1}{2}$&  $\frac{1}{12}$&  $\frac{1}{12}$&  $0$\\[5pt]
$(v_i, v_j) \in W_3$ &  $\frac{1}{3}$&  $0$&  $\frac{2}{3}$&  $-\frac{1}{9}$&  $0$&  $\frac{2}{9}$\\[5pt] \hline
\end{tabular}
\end{center}

\vspace{3mm}
As is seen, 
$\Delta_k(v_i, v_j) \ge 0, \, k=1, 2, 3$,
holds in each case and  
we can follow the same argument as 
the proof of Theorem \ref{Thm:four-eigenvalues} to
conclude that each function 
$h_{v_i, v_j}(t)$ is monotonically non-decreasing 
in $[0, \infty)$.  
Therefore, we have shown that 
the graph $X$ satisfies the MNHD. 
\end{proof}

\end{ex}

\begin{re}
\label{Cayley}
We should emphasize that the graph $X$ can be 
regarded as a Cayley graph of the symmetric group
$$
\mathfrak{S}_3=\{v_1=e, \, v_2=\sigma, \, v_3=\tau, \,
v_4=\sigma \tau, \, v_5=\tau \sigma, \, v_6=\sigma \tau \sigma\}
$$
over the finite set $\{1, 2, 3\}$ 
with a generating set $\{\sigma, \tau, \tau^2\}$, 
where $e$ denotes the identity permutation and 
$\sigma=(1, 2)$, $\tau=(1, 2, 3)$. 
\end{re}

\begin{ex}[non-regular case]
\normalfont
\label{Ex:2}
Let $X$ be a 6-wheel graph, that is, a finite graph 
given by connecting a certain vertex to 
all vertices of a 5-cycle 
(see {\bf (b)} in Fig. \ref{fig}).
The graph Laplacian of $X$ has four eigenvalues
$$
\lambda_0=0, \qquad 
\lambda_1=\frac{7-\sqrt{5}}{2}, \qquad 
\lambda_2=\frac{7+\sqrt{5}}{2}, \qquad 
\lambda_3=6.
$$
Although it is not regular, 
the graph $X$ also satisfies the MNHD. 

\begin{proof}
The graph Laplacian $L$ together with 
its square $L^2$ have on-diagonal components
given by 
$L(v_i, v_i)=3, \, L^2(v_i, v_i)=12$ if 
$i=1, 2, 3, 4, 5$ and 
$L(v_i, v_i)=5, \, L^2(v_i, v_i)=30$ if 
$i=6$, 
and have off-diagonal ones given by 
\begin{align}
L(v_i, v_j)&=0, & L^2(v_i, v_j)&=2, & 
&\text{if $v_i$ and $v_j$ are not adjacent}, 
\nn\\
L(v_i, v_j)&=-1, & L^2(v_i, v_i)&=-5, & 
&\text{if $v_i$ and $v_j$ are adjacent with $i, j \neq 6$}, 
\nn\\
L(v_i, v_j)&=-1, & L^2(v_i, v_j)&=-6, & 
&\text{otherwise}.
\nn
\end{align}
Then there are following four possibilities for 
the pair $(i, j)$:
\begin{enumerate}[(i)]
\item
$i, j=1, 2, 3, 4, 5$ and $v_i, v_j$ are not adjacent,
\item
$i, j=1, 2, 3, 4, 5$ and $v_i, v_j$
are adjacent,
\item
$i=1, 2, 3, 4, 5$ and $j=6$,
\item
$i=6$ and $j=1, 2, 3, 4, 5$.
\end{enumerate}
The values of $\Delta_k=\Delta_k(v_i, v_j), \, k=1, 2, 3$, 
and $\Delta_{k\ell}=\Delta_{k\ell}(v_i, v_j)$, 
$(k, \ell)=(1, 2), (1, 3), (2, 3)$, are listed as follows:

\vspace{3mm}
\begin{center}
\begin{tabular}{ccccccc}\hline
 & $\Delta_1$ & $\Delta_2$
 & $\Delta_3$ & $\Delta_{12}$
 & $\Delta_{13}$ & $\Delta_{23}$ \\[5pt]
 (i) & $\frac{\sqrt{5}+5}{10}$ & $\frac{5-\sqrt{5}}{10}$
 & 0 & $-\frac{2\sqrt{5}}{25}$ & $-\frac{\sqrt{5}+5}{300}$ & $-\frac{5-\sqrt{5}}{300}$ \\[5pt]
 (ii) & $\frac{5-\sqrt{5}}{10}$ & $\frac{\sqrt{5}+5}{10}$
 & 0 & $\frac{2\sqrt{5}}{25}$ & $-\frac{5-\sqrt{5}}{300}$
 & $-\frac{\sqrt{5}+5}{10}$ \\[5pt]
 (iii) & $\frac{2}{5}$ & $\frac{2}{5}$ & $\frac{1}{5}$
 & 0 & $\frac{1}{15}$ & $\frac{1}{15}$ \\[5pt]
 (iv) & 0 & 0 & 1 & 0 & 0 & 0 \\ \hline
\end{tabular}
\end{center}

\vspace{3mm}
We observe that $\Delta_k(v_i, v_j) \ge 0, \, k=1, 2, 3$,
holds in each case. 
Therefore, by following the same argument as 
the proof of Theorem \ref{Thm:four-eigenvalues}, 
we conclude that the function 
$h_{v_i, v_j}(t)$ is monotonically non-decreasing 
in $[0, \infty)$, which implies that 
the graph $X$
does satisfies the MNHD. 
\end{proof}
\end{ex}

\begin{figure}[h]
\begin{center}
\includegraphics[width=10cm]{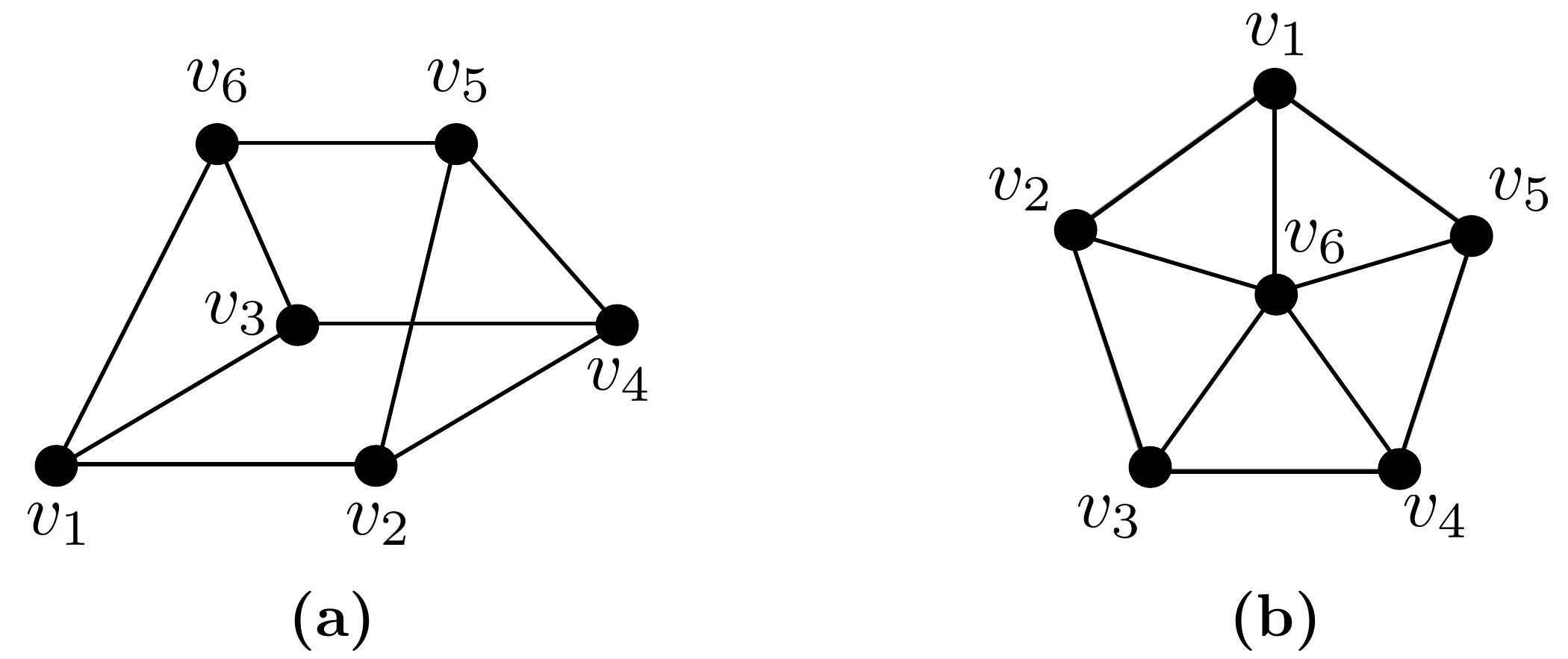}
\end{center}
\caption{Finite graphs with four Laplacian eigenvalues.}
\label{fig}
\end{figure}

In the present paper, 
we have proved that every 
finite, regular and bipartite graph
with four distinct Laplacian eigenvalues 
satisfies the MNHD. 
However, in view of Examples \ref{Ex:1}
and \ref{Ex:2}, 
our results may be extended to more
general cases such as at least non-bipartite cases.
Indeed, we expect that 
the MNHD truly holds
for all finite regular graphs with four Laplacian 
eigenvalues, which is left as an interesting
open problem. 
Simultaneously, it is also interesting to 
investigate the cases 
of non-regular graphs. 

Recall that every finite abelian Cayley graph
satisfies the MNHD, which was shown in Price \cite{Price}.
It is known that the symmetric group $\mathfrak{S}_3$
is a typical example of finite groups which
are not abelian but solvable. 
As we mentioned in Example \ref{Ex:1}
and Remark \ref{Cayley}, 
the Cayley graph of 
$\mathfrak{S}_3$ enjoys the MNHD. 
On the other hand, the finite lamplighter group, 
which is also known as an example of finite 
solvable groups, does not satisfy the MNHD
(see Regev--Shinkar \cite{RS}). 
By taking these circumstances into account, 
one may ask what kinds of finite 
non-commutative groups satisfy the MNHD. 
In particular, we left the following as an open problem. Does every finite {\it nilpotent} Cayley 
graph satisfy the MNHD?

\vspace{2mm}
\noindent
{\bf Acknowledgements}:
The second-named author would like to thank
Professor Ryokichi Tanaka for informing him
of the results in Regev--Shinkar \cite{RS} and Price \cite{Price}
when he visited Tohoku University in June 2017, 
and for giving him valuable comments.
This work is supported by JSPS KAKENHI
Grant Number 19K23410.

%
%




\end{document}